\documentclass[11pt]{amsart}
\usepackage[utf8]{inputenc}
\usepackage{mathtools, tikz, amsfonts, amsmath, amssymb, amsthm, comment, enumitem}
\usepackage[margin=1in]{geometry}
\usepackage[colorlinks=true, citecolor=blue]{hyperref}

\newcommand{\cplx}{\mathbb{C}}
\newcommand{\oscr}{\mathcal{O}}
\newcommand{\hfr}{\mathfrak{h}}
\newcommand{\ints}{\mathbb{Z}}
\newcommand{\PP}{\mathcal{P}}
\newcommand{\wc}{\mathfrak{wc}}

\newtheorem{theorem}{Theorem}[section]

\newtheorem{definition}[theorem]{Definition}
\newtheorem{lemma}[theorem]{Lemma}

\newtheorem{proposition}[theorem]{Proposition}
\newtheorem{remark}[theorem]{Remark}

\begin{document}

\title[Finite-dimensional representations of rational Cherednik algebras of type D]{Towards a classification of finite-dimensional representations of rational Cherednik algebras of type D}

\author{Seth Shelley-Abrahamson}
\address[Seth Shelley-Abrahamson]{Department of Mathematics, Massachusetts Institute of Technology}
\email{sethsa@alumni.stanford.edu}

\author{Alec Sun}
\address[Alec Sun]{Phillips Exeter Academy}
\email{sundogx@gmail.com}

\begin{abstract}
Using a combinatorial description due to Jacon and Lecouvey of the wall crossing bijections for cyclotomic rational Cherednik algebras, we show that the irreducible representations $L_c(\lambda^\pm)$ of the rational Cherednik algebra $H_c(D_n, \cplx^n)$ of type $D$ for symmetric bipartitions $\lambda$ are infinite dimensional for all parameters $c$.  In particular, all finite dimensional irreducible representations of rational Cherednik algebras of type $D$ arise as restrictions of finite-dimensional irreducible representations of rational Cherednik algebras of type $B$.
\end{abstract}

\maketitle

\section{Introduction}  Rational Cherednik algebras were introduced by Etingof and Ginzburg \cite{EG}.  For any finite complex reflection group $W$ with reflection representation $\hfr$ and complex $W$-invariant function $c : S \rightarrow \cplx$ on the set of reflections $S \subset W$ of $W$, the rational Cherednik algebra $H_c(W, \hfr)$ is an infinite-dimensional associative $\cplx$-algebra deforming the semidirect product algebra $H_0(W, \hfr) = \cplx W \ltimes D(\hfr)$, where $D(\hfr)$ denotes the algebra of polynomial differential operators on the vector space $\hfr$.  Ginzburg, Guay, Opdam, and Rouquier \cite{GGOR} introduced a certain category $\oscr_c(W, \hfr)$ of representations of $H_c(W, \hfr)$ analogous to the BGG categories $\oscr$ appearing in the representation theory of complex semisimple Lie algebras.  The simple objects in $\oscr_c(W, \hfr)$ are naturally labeled by the irreducible representations of the underlying reflection group $W$, and for generic parameters $c$ this category is semisimple.  In this way $\oscr_c(W, \hfr)$ can be regarded as a deformation of the category of finite-dimensional complex representations of $W$.  The rational Cherednik algebras $H_c(W, \hfr)$ and the categories $\oscr_c(W, \hfr)$ have been studied in great detail and have deep connections with many other important objects in mathematics.  However, several very basic questions about these algebras remain open problems, including in many cases the classification and description of their finite-dimensional irreducible representations.

For a positive integer $n \geq 4$, let $B_n := S_n \ltimes (\ints/2\ints)^n$ denote the finite real reflection group of type $B$ and rank $n$, acting in its (complexified) reflection representation $\cplx^n$ by permutations and sign changes of the coordinates.  Let $D_n \subset B_n$ denote the index 2 reflection subgroup of $B_n$ of type $D$, consisting of those elements acting on $\cplx^n$ with an even number of sign changes.  The irreducible complex representations of $B_n$ are naturally labeled by bipartitions $\lambda = (\lambda^1, \lambda^2)$ of $n$; we denote the irreducible representation of $B_n$ labeled by the bipartition $\lambda$ by $V_\lambda$ (see \cite{GP}, for example, for the standard constructions).  The irreducible representations of $D_n$ are described similarly.  In particular, the representation $V_{(\lambda^1, \lambda^2)}$ of $B_n$ remains irreducible after restriction to $D_n$ when $\lambda^1 \neq \lambda^2$, i.e. when the bipartition $\lambda = (\lambda^1, \lambda^2)$ is not symmetric, and in that case $V_{(\lambda^1, \lambda^2)}$ is isomorphic to $V_{(\lambda^2, \lambda^1)}$ as a representation of $D_n$.  When $\lambda$ is symmetric, the restriction of $V_\lambda$ to $D_n$ splits into a direct sum of two non-isomorphic irreducible representations $V_\lambda^+$ and $V_\lambda^-$.  The representations $V_\lambda$ for non-symmetric $\lambda$ along with the representations $V_\lambda^\pm$ for symmetric $\lambda$ form a complete list of the irreducible representations of $D_n$ up to isomorphism, and the only isomorphisms among pairs of these representations are the isomorphisms $V_{(\lambda^1, \lambda^2)} \cong V_{(\lambda^2, \lambda^1)}$ for $\lambda^1 \neq \lambda^2$.  The irreducible objects in the categories $\oscr_c(B_n, \cplx^n)$ and $\oscr_c(D_n, \cplx^n)$ therefore are labeled in the same way; to an irreducible representation $V$ of $B_n$ or $D_n$, there is a corresponding irreducible representation $L_c(V)$ in $\oscr_c(B_n, \cplx^n)$ or $\oscr_c(D_n, \cplx^n)$, respectively.

In this paper, we show that the irreducible representations in $\oscr_c(D_n, \cplx^n)$ of the form $L_c(\lambda^\pm)$ for symmetric bipartitions $\lambda = (\lambda^1, \lambda^2) = (\lambda^2, \lambda^1)$ are infinite dimensional for all parameters $c$.  We reduce this problem to a combinatorial statement in terms of the wall crossing bijections for rational Cherednik algebras introduced by Losev \cite{L}, and we prove the corresponding combinatorial statement by an inductive argument using a concrete description of the wall crossing bijections provided recently by Jacon and Lecouvey \cite{JL}.  In particular, this implies that all finite-dimensional representations of rational Cherednik algebras of type $D$ arise as restrictions of finite-dimensional irreducible representations of rational Cherednik algebras of type $B$, and hence also that the orbits of these finite-dimensional representations with respect to the action by twists by the normalizer $N_W(D_n)$ of $D_n$ in any larger reflection group $W$ containing $D_n$ as a parabolic subgroup are trivial.

\subsection*{Acknowledgements} This paper represents the results of a research project undertaken by the second author, for which the first author served as a graduate student mentor, as part of MIT's Program for Research in Mathematics, Engineering, and Science for High School Students (PRIMES-USA).  We are very grateful to the PRIMES program for making this project possible.

\section{Background and Definitions}

\subsection{Rational Cherednik Algebras}  Let $W$ be a finite complex reflection group with reflection representation $\hfr$.  Let $S \subset W$ be the set of reflections in $W$, and let $c : S \rightarrow \cplx$ be a $W$-invariant function.  For each $s \in S$, let $\alpha_s \in \hfr^*$ be an eigenvector of $s$ with eigenvalue $\lambda_s \neq 1$, and let $\alpha_s^\vee \in \hfr$ be an eigenvector of $s$ with eigenvalue $\lambda_s^{-1}$, normalized so that $\langle \alpha_s, \alpha_s^\vee \rangle = 2$, where $\langle \cdot, \cdot \rangle$ denotes the natural evaluation pairing $\hfr^* \times \hfr \rightarrow \cplx$.  To this data one can associate the rational Cherednik algebra $H_c(W, \hfr)$, defined by generators and relations as follows \cite{EG}: $$H_c(W, \hfr) =  \frac{\cplx W \ltimes T(\hfr^* \oplus \hfr)}{\langle[x, x'] = [y, y'] = 0, \ \ [y, x] = \langle y, x\rangle - \displaystyle\sum_{s \in S} c_s \langle \alpha_s, y\rangle \langle x, \alpha_s^\vee \rangle s \ \ \forall x \in \hfr^*, y \in \hfr\rangle}.$$  The natural map $$\cplx[\hfr] \otimes \cplx W \otimes \cplx[\hfr^*] \rightarrow H_c(W, \hfr)$$ induced by multiplication is an isomorphism of vector spaces.  The category $\oscr_c(W, \hfr)$ is defined \cite{GGOR} to be the full subcategory of $H_c(W, \hfr)$-modules that are finitely generated over $\cplx[\hfr]$ and on which $\hfr \subset \cplx[\hfr^*]$ acts locally nilpotently.

The irreducible representations in $\oscr_c(W, \hfr)$ are naturally labeled by the irreducible representations of the group $W$.  In particular, let $V_\lambda$ be an irreducible representation of $W$.  The action of $W$ on $V_\lambda$ can be extended to $\cplx[\hfr^*] \rtimes \cplx W$ by letting $\hfr$ act by 0, and the \emph{standard module} $\Delta_c(V_\lambda) := H_c(W, \hfr) \otimes_{\cplx[\hfr^*] \rtimes \cplx W} V_\lambda$ is a representation in $\oscr_c(W, \hfr)$.  The representation $\Delta_c(V_\lambda)$ has a unique irreducible quotient, denoted $L_c(V_\lambda)$.  The irreducible representations $L_c(V_\lambda)$ appearing in this way are nonisomorphic and form a complete list, up to isomorphism, of the irreducible representations in $\oscr_c(W, \hfr)$.

\subsection{Algebras of Types $B$ and $D$}  For a positive integer $n \geq 4$, let $B_n$ denote the real reflection group of type $B$ and rank $n$ and let $D_n$ denote the real reflection group of type $D$ and rank $n$.  Recall that $B_n$ is isomorphic to the semidirect product $S_n \rtimes (\ints/2\ints)^n$, where $S_n$ is the symmetric group acting on the product $(\ints/2\ints)^n$ by permutations of the coordinates.  $B_n$ acts on its complexified reflection representation $\cplx^n$ by permutations and sign changes of the coordinates.  The group $D_n$ embeds naturally as an index 2 subgroup of $B_n$, consisting of those elements which act on $\cplx^n$ with an even number of sign changes.  The reflections through the hyperplanes $z_i = \pm z_j$ in $\cplx^n$ generate $D_n$, and these reflections form a complete conjugacy class in $B_n$. The remaining conjugacy class of reflections in $B_n$ consists of those reflections through the hyperplanes $z_i = 0$.  In particular, a parameter $c$ for the rational Cherednik algebra $H_c(D_n, \cplx^n)$ is determined by its value on a reflection through any of the hyperplanes $z_i = \pm z_j$, and therefore the parameter $c$ can be identified with a complex number $c \in \cplx$ in this way.  Similarly, a parameter for the rational Cherednik algebra of type $B$ can be identified with a pair of complex numbers $(c_1, c_2) \in \cplx^2$, with $c_1$ specifying the value of the parameter on the reflections through the hyperplanes $z_i = \pm z_j$ and $c_2$ specifying the value of the parameter on the reflections through the hyperplanes $z_i = 0$.  It follows immediately from the definition of the rational Cherednik algebra and its description as a vector space given in the previous section that for any $c \in \cplx$ the algebra $H_c(D_n, \cplx^n)$ embeds naturally as a subalgebra of $H_{c, 0}(B_n, \cplx^n)$.

Recall that a \emph{partition} of a nonnegative integer $n \geq 0$ is a weakly decreasing, possibly empty, list $\lambda_1 \geq \lambda_2 \geq \cdots \geq \lambda_l > 0$ of positive integers with sum $\sum_i \lambda_i$ equal to $n$.  We will often denote a partition $\lambda$ by the list of its parts, i.e. $\lambda = (\lambda_1, ..., \lambda_l)$, and we denote the sum of the parts of $\lambda$, i.e. its \emph{size}, by $|\lambda| := \sum_i \lambda_i$.  Let $\PP(n)$ denote the set of partitions of $n$, and let $\PP = \bigcup_{n = 0}^\infty \PP(n)$ denote the set of all partitions.  A \emph{bipartition} of $n$ is an ordered pair of partitions $\lambda = (\lambda^1, \lambda^2)$ such that $|\lambda| := |\lambda^1| + |\lambda^2| = n$.  Let $\PP^2(n)$ denote the set of bipartitions of $n$, and let $\PP^2 := \bigcup_{n = 0}^\infty \PP^2(n)$ denote the set of all bipartitions.

There is a natural bijection (see, e.g. \cite{GP}) $\lambda \mapsto V_\lambda$ between the set $\PP^2(n)$ of bipartitions of $n$ and the set of irreducible representations of $B_n$.  In particular, the correspondence $\lambda \mapsto L_{c_1, c_2}(\lambda) := L_{c_1, c_2}(V_\lambda)$ gives a bijection between $\PP^2(n)$ and the set of irreducible representations in $\oscr_{c_1, c_2}(B_n, \cplx^n)$.  The irreducible representations of $D_n$ are described similarly.  Call a bipartition $\lambda = (\lambda^1, \lambda^2)$ \emph{symmetric} if $\lambda^1 = \lambda^2$, and not symmetric otherwise.  The representation $V_\lambda$ of $B_n$ remains irreducible after restriction to $D_n$ when $\lambda$ is not symmetric.  When $\lambda$ is symmetric, however, $V_\lambda$ splits upon restriction to $D_n$ into a direct sum of two non-isomorphic irreducible representations $V_\lambda^+$ and $V_\lambda^-$ of $D_n$.  For symmetric bipartitions $\lambda$, let $L_c(\lambda^+) := L_c(V_\lambda^+)$ and $L_c(\lambda^-) := L_c(V_\lambda^-)$ denote the associated irreducible representations in $\oscr_c(D_n, \cplx^n)$.  The following theorem is the main result of this paper:

\begin{theorem} \label{main-theorem} The representation $L_c(\lambda^\pm)$ of $H_c(D_n, \cplx^n)$ is infinite-dimensional for all symmetric bipartitions $\lambda$ of $n$ and for all parameters $c \in \cplx$.\end{theorem}

We will prove Theorem \ref{main-theorem} by using the combinatorics of wall crossing bijections, recalled in the following section.  The algebra $H_c(D_n, \cplx^n)$ does not have any nonzero finite-dimensional representations unless $c$ is a non-integral rational number.  After twisting by the sign character of $D_n$, one may assume that $c$ is of the form $c = r/e$ for a positive integer $e$ and negative integer $r < 0$ relatively prime to $e$.  As there are isomorphisms of vector spaces $L_{c, 0}(\lambda) \cong L_c(\lambda^+) \oplus L_c(\lambda^-)$ and $L_c(\lambda^+) \cong L_c(\lambda^-)$, it suffices to show that the representation $L_{c, 0}(\lambda)$ of $H_{c, 0}(B_n, \cplx^n)$ is infinite-dimensional.  By \cite[Lemma 4.2]{L}, $H_{-r/e, 0}(B_n, \cplx^n)$ has no nonzero finite-dimensional representations unless $e$ is an even integer, so it suffices to reduce further to the case $c = r/e$ with $e \in 2\ints^{> 0}$ a positive even integer and $r < 0$ a negative integer relatively prime to $e$. 

\subsection{Wall-Crossing Bijections in Type $B$}\label{2.3} \label{reduction-section} In \cite{L}, Losev introduced bijections $\wc_{c'\leftarrow c}$, depending on certain parameters $c$ and $c'$ for the rational Cherednik algebra of type $B$, on the set of bipartitions $\PP^2(n)$ of fixed rank $n$ (and, more generally, for multipartitions in the cyclotomic case).  We will recall in this section only the necessary facts about these bijections needed to prove Theorem \ref{main-theorem}.

Fix a positive even integer $e$ as in the previous section.  To any pair of integers $s = (s_1, s_2) \in \ints^2$, which we refer to as a \emph{charge}, we can associate a parameter $(c_1, c_2)$ for a rational Cherednik algebra of type $B$ by setting $c_1 = -1/e$ and $c_2 = (s_2 - s_1 - \frac{e}{2})/e.$  Note that the pair $s = (0, e/2)$ corresponds to the parameter $c = (-1/e, 0)$ relevant to Theorem \ref{main-theorem}.  Let $s = (s_1, s_2) \in \ints^2$ be a pair of integers, let $s' = (s_1, s_2 + e)$, and let $c$ and $c'$ be the parameters for the rational Cherednik algebra of type $B$ corresponding to $s$ and $s'$, respectively.  Losev's \emph{wall crossing bijection} $\wc_{c'\leftarrow c}$ associated to these parameters is a bijection $\wc_{c'\leftarrow c} : \PP^2(n) \rightarrow \PP^2(n)$ with the property that for all bipartitions $\lambda \in \PP^2(n)$ the representation $L_c(\lambda)$ is finite-dimensional if and only if the representation $L_{c'}(\wc_{c'\leftarrow c}(\lambda))$ is finite-dimensional \cite[Corollary 2.13]{L}.  Furthermore, if $s_1 < s_2 - n$, then for all bipartitions $\lambda = (\lambda^1, \lambda^2) \in \PP^2(n)$ the representation $L_c(\lambda)$ is infinite-dimensional unless $\lambda_2 = \emptyset$ \cite[Proposition 5.3]{L} (note that here we use the opposite convention of Losev on the sign of the parameter $c_1 = -1/e$).  This gives a strategy for proving Theorem \ref{main-theorem}: show that for any symmetric bipartition $\lambda$, the bipartition obtained by successive applications of wall crossing bijections has nonempty second component.  By the description of the wall crossing bijections given in \cite[Section 4]{L}, the wall crossing bijections that must be applied to a bipartition $\lambda$ in the cases $c_1 = -1/e$ and $c_1 = r/e$ for $r < 0$ a negative integer relatively prime to $e$ are the same, so we may assume $c_1 = -1/e$ and prove Theorem \ref{main-theorem} in this case.

A \emph{charged bipartition} is the data of a bipartition $\lambda$ along with a \emph{charge} $s = (s_1, s_2) \in \ints^2$, denoted $|\lambda, s\rangle$.  There is a natural correspondence between charged bipartitions and another type of combinatorial objects called \emph{two-row tableaux}, which are defined as follows.  Let $|\lambda, s\rangle$ be a charged bipartition with $|\lambda|=n$ and $s=(s_1,s_2).$ Let $d\ge |s_1-s_2|$ be the minimal integer such that $\lambda^1_{d+1+s_1-s_2}=\lambda^2_{d+1}=0$ if $s_2\ge s_1$ or such that $\lambda^1_{d+1}=\lambda^2_{d+1+s_1-s_2} = 0$ if $s_1\ge s_2.$ To the charged bipartition $|\lambda, s\rangle,$ associate the following \emph{two-row tableau} (two row array of integers)
\hfill \break
\begin{center}
$T(\lambda^1, \lambda^2)=$
\begin{tabular}{c c c c c c}
$s_2-d+\lambda^2_{d+1}$ & $\cdots$ & $\cdots$ & $\cdots$ & $s_2-1+\lambda^2_2$ & $s_2+\lambda^2_1$\\
$s_2-d+\lambda^1_{d+1+s_1-s_2}$ & $\cdots$ & $s_1+\lambda^1_1.$
\end{tabular}
\end{center}
\hfill \break
when $s_2\ge s_1$ and the following two-row tableau
\hfill \break
\begin{center}
$T(\lambda^1, \lambda^2)=$
\begin{tabular}{c c c c c c}
$s_1-d+\lambda^2_{d+1+s_1-s_2}$ & $\cdots$ & $s_2+\lambda^2_1.$\\
$s_1-d+\lambda^1_{d+1}$ & $\cdots$ & $\cdots$ & $\cdots$ & $s_1-1+\lambda^1_2$ & $s_1+\lambda^1_1$\\
\end{tabular}
\end{center}
\hfill \break
when $s_1\ge s_2.$ We denote the top row of $S(\lambda^1, \lambda^2)$ by $L_2$ and the bottom row by $L_1$, so that $S(\lambda^1, \lambda^2) = \binom{L_2}{L_1}$.  Note that the top row corresponds to $\lambda^2$ and the bottom row corresponds to $\lambda^1.$ By the definition of $d,$ the entries in the leftmost column of $S(\lambda^1, \lambda^2)$ are equal to $s_2-d$ when $s_2\ge s_1$ and $s_1-d$ when $s_1 \ge s_2.$  Note that the original charged bipartition $|\lambda, s\rangle$ can be recovered from its associated two-row tableau $S(\lambda^1, \lambda^2)$.  In particular, we can recover the values $s_1, s_2$ and $d$ from the lengths of $L_1$ and $L_2$ and hence also the charged bipartition $|\lambda, s\rangle$. 

We now recall Jacon and Lecouvey's combinatorial procedure \cite{JL} for computing the wall crossing bijections $\wc_{c' \leftarrow c}$.  For each charge $s = (s_1, s_2) \in \ints^2$, there is a bijection $\Phi^\infty_{(s_1, s_2)} : \PP^2(n) \rightarrow \PP^2(n)$ defined in terms of associated two-row tableaux as follows.  Let $\lambda = (\lambda^1, \lambda^2) \in \PP^2(n)$ be a bipartition of $n$.  Let $S(\lambda^1, \lambda^2) = \binom{L_2}{L_1}$ be the two-row tableau associated to the charged bipartition $|\lambda, s\rangle$.  We construct a new two-row tableau $\binom{L_2'}{L_1'}$ from the two-row tableau $\binom{L_2}{L_1}$ in a combinatorial manner; the bijection $\Phi^\infty_{(s_1, s_2)}$ is then defined by setting $\Phi^\infty_{(s_1, s_2)}(\lambda) = \lambda'$, where $\lambda' \in \PP^2(n)$ is the bipartition such that the charged bipartition $|\lambda', s\rangle$ has associated two-row tableau $\binom{L_2'}{L_1'}$.  The transformation of the two-row tableau $\binom{L_2}{L_1}$ into the two-row tableau $\binom{L_2'}{L_1'}$ consists of two steps: (1) entries in the bottom row of the tableau are swapped with some entries in the top row of the tableau, and (2) each of the rows, after the swapping procedure, is sorted into increasing order.

We first describe procedure (1). Suppose that $s_2\ge s_1,$ which will be true for all charges we consider in this paper (the procedure for $s_1\ge s_2$ is analogous). Consider $x_1 := \min\{t\in L_1\}.$ Associate to $x_1$ the integer $y_1\in L_2$ defined by $$y_1 :=
\begin{cases}
\max\{z\in L_2|z\le x_1\} \text{ if } \min\{z\in L_2\}\le x_1,\\
\max\{z\in L_2\} \text{ otherwise }\\
\end{cases}$$
Then repeat the same procedure with the arrays $L_2-\{y_1\}$ and $L_1-\{x_1\}$, producing elements $x_2 \in L_1-\{x_1\}$ and $y_2 \in L_2-\{x_2\}.$ By induction this yields a sequence $\{y_1,\ldots,y_{d+1+s_1-s_2}\}\in L_2.$

Now we describe procedure (2).  Let $L_1'$ be the list of integers obtained by sorting, in increasing order, the set $\{y_1,\ldots,y_{d+1+s_1-s_2}\}\in L_2$, and let $L_2'$ be the list of integers obtained by sorting the set $L_2-\{y_1,\ldots,y_{d+1+s_1-s_2}\}+L_2$ in increasing order.  These lists define the two-row tableau $\binom{L_2'}{L_1'}$ and hence the bipartition $\lambda' = \Phi^\infty_{(s_1, s_2)}(\lambda)$.

Let $c = (-1/e, c_2)$ and $c' = (-1/e, c_2')$ be the parameters for the rational Cherednik algebra of type $B$ with first component $-1/e$ associated to the charges $(s_1, s_2)$ and $(s_1, s_2 + e)$, respectively.  In \cite[Theorem 4.10]{JL}, Jacon and Lecouvey prove that the bijection $\Phi^\infty_{(s_1, s_2)}$ described combinatorially above equals Losev's wall crossing bijection $\wc_{c' \leftarrow c}$.

We will refer to the procedures (1) and (2) as \emph{swapping} and \emph{sorting}, respectively.

\begin{definition}
Let $e > 1$ be a positive integer, and let $s = (s_1, s_2) \in \ints^2$ be a charge. Let $\Theta_{e, s}$ denote the iterated composition of wall-crossing bijections $$\Theta_{e, s} := \prod_{k = 0}^\infty \Phi^\infty_{(s_1, ke + s_2)} := \cdots \circ \Phi^\infty_{(s_1, Ne + s_2)} \circ \cdots \circ \Phi^\infty_{(s_1, e + s_2)} \circ \Phi^\infty_{(s_1, s_2)} : \mathcal{P}^2(n) \rightarrow \mathcal{P}^2(n).$$
\end{definition}

\begin{remark} Note that the bijection $\Phi^\infty_{(s_1, ke + s_2)}$ fixes $\lambda$ when $|\lambda| < ke + s_2 - s_1$.  It follows that $\Theta_{e, s}$ acts as a finite product on any bipartition and is a well-defined operator on bipartitions.\end{remark}

By the discussion in Section \ref{reduction-section}, to prove Theorem \ref{main-theorem} it suffices to prove the following:

\begin{theorem} \label{combinatorial-theorem} Let $\lambda \in \PP^2(n)$ be a symmetric bipartition of $n$.  For any positive even integer $e$, the operator $\Theta_{e,(0,\frac{e}{2})}$ leaves the first part $\lambda^2_1$ of the second component $\lambda^2$ of $\lambda$ invariant.\end{theorem}

The rest of this paper is dedicated to proving Theorem \ref{combinatorial-theorem}.

\section{Infinite-Dimensionality of the Representations $L_c(\lambda^\pm)$}

\begin{definition}
Let $\wc$ denote the bijection on the set of two-row tableaux corresponding to the bijection $|\lambda, s\rangle \mapsto |\Phi^\infty_{s,e}(\lambda), s\rangle$ on the set of charged bipartitions.
\end{definition}

In the language of the previous section, $\wc$ is computed by performing the swapping procedure (1) followed by the sorting procedure (2).

\begin{definition}
Let $\Delta_e$ denote the bijection on the set of two-row tableaux which sends the two-row tableau $\binom{a}{b}$ representing a charged bipartition $|\lambda,(s_1,s_2)\rangle$ to the two-row tableau $\Delta_e\left(\binom{a}{b}\right)$ representing the charged bipartition $|\lambda,(s_1,s_2+e)\rangle.$
\end{definition}

Fix a symmetric bipartition $\lambda$ of $n$ and a positive even integer $e>0.$  Let $\binom{a_0}{b_0}$ be the two-row tableau associated to $|\lambda,(0,\frac{e}{2})\rangle.$ For $k\ge 0$, let $\binom{\alpha_k}{\beta_k}=\mathfrak{wc}\left(\binom{a_k}{b_k}\right)$ and let $\binom{a_{k + 1}}{b_{k + 1}}=\Delta_e\left(\binom{\alpha_{k}}{\beta{k}}\right).$
Label the entries of $\binom{\alpha_{k}}{\beta_{k}}$ by $\binom{\alpha_k}{\beta_k} = $
\begin{center}
\begin{tabular}{c c c c}
$\alpha_k^1$&$\alpha_k^2$&$\cdots$&$\alpha_k^{d+ke+\frac{e}{2}}$\\
$\beta_k^1$&$\cdots$&$\beta_k^d.$
\end{tabular}
\end{center}
and similarly label the entries of $\binom{a_{k+1}}{b_{k+1}}$ by $\binom{a_{k+1}}{b_{k+1}}=$
\begin{center}
\begin{tabular}{c c c c}
$a_{k+1}^1$&$\cdots$&$\cdots$&$a_{k+1}^{d+(k+1)e+\frac{e}{2}}$\\
$b_{k+1}^1$&$\cdots$&$b_{k+1}^d.$
\end{tabular}
\end{center}

\begin{remark} \label{delta-e}
Notice that it is easy to describe the entries of $\binom{a_{k + 1}}{b_{k + 1}}=\Delta_e\left(\binom{\alpha_{k}}{\beta{k}}\right)$ in terms of the entries of $\binom{\alpha_k}{\beta_k}$.  In particular, we have $b_{k + 1}^i = \beta_k^i$ for $1 \leq i \leq d$, $a_{k + 1}^i = \beta_k^1 + i - 1$ for $1 \leq i \leq e$, and $a_{k + 1}^i = \alpha^{i - e}_k + e$ for $i > e$.
\end{remark}

Notice also that for all sufficiently large $k$, the charged bipartition associated to the two-row tableau $\binom{a_k}{b_k}$ is $|\Theta_{e, s}(\lambda), (0, ke + \frac{e}{2})\rangle.$

Recall from Section \ref{2.3} that the bijection $\mathfrak{wc}$ maps $\binom{a_k}{b_k}$ to $\binom{\alpha_k}{\beta_k}$ by swapping each integer $b^j_k$ with the maximal unused integer $a^{i(j)}_k$ from the top row for which $a^{i(j)}_k\le b^j_k$, if such an integer exists, and otherwise with the largest entry $a^{i_j}_k$ in the top row that has not already been swapped.

\begin{definition}
In the former case, we write $S(b^j_k)=a^{i(j)}_k$ and $S(a^{i(j)}_k)=b^j_k.$ We say that the entries $b^j_k$ and $a^{i(j)}_k$ are \emph{swapped}.  In the latter case, then we call $a^{i_j}_k$ a \emph{cycle}.\end{definition}

\begin{definition}
Define a \emph{hole} in a two-row tableau $\binom{a}{b}$ to be an integer $a^i$ in the top row that is not swapped with an entry in the bottom row in the procedure defining $\wc$.
\end{definition}

There are $ke+\frac{e}{2}$ holes in $\binom{a_k}{b_k}$. Note that $\mathfrak{wc}$ acts on the two row tableau $\binom{a_k}{b_k}$ by exchanging, in an order-preserving manner, the \emph{ordered} set of all $b^1_k<\cdots<b^d_k$ in the bottom row with a uniquely determined ordered set of $d$ integers $a^{i_1}_k<\cdots<a^{i_d}_k$ from the top row, and then sorting the two rows. Correspondingly, we make the following definition:

\begin{definition} \label{M}
Denote $M(b^j_k)=a^{i_j}_k$ and $M(a^{i_j}_k)=b^j_k.$
\end{definition}

\begin{remark} \label{meta-increasing} Note that when $M(a_k^x)$ and $M(a_k^y)$ are defined, we have $M(a_k^x) < M(a_k^y)$ when $x < y$.\end{remark}

The following definitions will be convenient in the proof:

\begin{definition}
For each nonnegative integer $k \geq 0$, define statements $A_k, B_k, C_k$ and $D_k$ as follows:
\begin{enumerate}
\item Let $A_k$ be the statement that in $\binom{a_k}{b_k}$ we have $b^j_k\ge a^j_k$ for all $1\le j\le d.$ 

\item Let $B_k$ be the statement that there are no cycles in the tableau $\binom{a_k}{b_k}$.

\item Let $C_k$ be the statement that if $M(b^j_k)=a^{i_j}_k$ then $b^j_k\ge a^{i_j}_k.$

\item Let $D_k$ be the statement that $a^i_k=\alpha^i_k$ for all holes $a^i_k.$
\end{enumerate}
\end{definition}

\begin{lemma}\label{reduction-lemma} If $B_k$ is true for all integers $k \geq 0$, then Theorem \ref{combinatorial-theorem} holds.\end{lemma}

\begin{proof} If $B_k$ is true for all integers $k \geq 0$, then the last entry of the top row of $\binom{a_k}{b_k}$ is a hole, and in particular for all $k \geq 0$ the last entry of the first row of $\binom{a_{k + 1}}{b_{k + 1}}$ is obtained from the last entry of the first row of $\binom{a_k}{b_k}$ by adding $e$.  As the charge associated to the two-row tableau $\binom{a_k}{b_k}$ is $(0, ke + e/2)$, Theorem \ref{combinatorial-theorem} follows immediately from the definition of the correspondence between two-row tableaux and charged bipartitions.\end{proof}

Therefore, to prove Theorem \ref{combinatorial-theorem}, we will show that $B_k$ holds for all $k \geq 0$.  This will be achieved by an inductive argument involving the statements $A_k, B_k, C_k$ and $D_k$ defined above.

\begin{remark}
$A_0$ is true by inspection of the two-row tableau associated to $|\lambda, (0, \frac{e}{2})\rangle$.
\end{remark}

\begin{proposition}
$A_k$ implies $B_k.$
\end{proposition}
\begin{proof}
Suppose $A_k$ is true.  Then $b^j_k \geq a_k^j \geq a^{j'}_k$ for all positive integers $j'\le j.$  It follows that when performing the $\wc$ procedure, at the $j^{th}$ step at most $j - 1$ of the $a_k^{j'}$ for $j' \leq j$ have already been paired with entries in the bottom row, so at least one $a^{j'}_k \leq b_k^j$ for $j' \leq j$ is available.  Statement $B_k$ follows.
\end{proof}

\begin{proposition}
$B_k$ implies $C_k.$
\end{proposition}
\begin{proof}  Suppose $B_k$ is true.  Consider the integer $j'$ for which $S(b^j_k)=a^{i_{j'}}_k.$  We consider two cases:

\begin{description}
\item[Case 1] $j'\ge j.$ Then $b^j_k\ge a^{i_{j'}}_k\ge a^{i_j}_k.$

\item[Case 2] $j'<j.$ Then by the pigeonhole principle there exist $j_1<j$ and $j_2>j'$ for which $S(b^{j_1}_k)=a^{i_{j_2}}_k.$ Then we have $b^j_k\ge b^{j_1}_k\ge a^{i_{j_2}}_k\ge a^{i_{j'}}_k.$\end{description}\end{proof}

\begin{lemma}
\label{hole-invariance}
$C_k$ implies $D_k.$\end{lemma}
\begin{proof}
First we prove that the first hole $a^{h_1}_k$ remains in index $h_1$ after $\wc$.  Represent all $d$ swaps arising while performing $\wc$ on $\binom{a_k}{b_k}$ as two-way arrows connecting $b^j_k$ to $a^i_k.$ Let $r\ge 1$ be the minimum integer for which $a^{h_1+r}_k$ is not a hole. We need to prove that $M(a^{h_1+r}_k)\ge a^{h_1+r-1}_k$ and $M(a^{h_1-1}_k)\le a^{h_1}_k,$ which will imply that the position of $a^{h_1}_k$ will not be changed during sorting.

Assuming $C_k,$ $M(a^{h_1+r}_k)\ge a^{h_1+r}_k \ge a^{h_1+r-1}_k.$ Hence the first inequality is true.

Now we prove the second inequality. By Remark \ref{meta-increasing} we have $M(a^{h_1-1}_k)=b^{h_1-1}_k.$ Now assume for contradiction that $b^{h_1-1}_k>a^{h_1}_k.$ Then $b^{h_1-1}_k$ is swapped with $a^l_k$ for some $l>h_1,$ otherwise $a^{h_1}_k$ would not be a hole.

Consider the swaps of the entries $a^1_k,\ldots,a^{h_1-1}_k.$  Use the pigeonhole principle on $a^1_k,\ldots,a^{h_1-1}_k$ to conclude that there exist $x<h_1$ and $y\ge h_1$ for which $S(b_k^y)=a_k^x.$  But now we have $b^y_k>b^{h_1-1}_k\ge a^l_k>a^{h_1}_k,$ which, as $a_k^{h_1}$ is greater than $a_k^x$ and is never swapped during the $\wc$ procedure, contradicts that $S(b_k^y) = a_k^x.$  Hence we conclude $b^{h_1-1}_k\le a^{h_1}_k.$  We conclude as a corollary that $M(b^j_k)=a^j_k$ for all $j<h_1$ since $a^{h_1}_k>b^j$ for all $j<h_1.$

Let $\binom{a_k'}{b_k'}$ denote the tableau obtained from $\binom{a_k}{b_k}$ by including only the first $h_1 + r - 1$ entries of the first row $a_k$ and only the first $h_1 - 1$ entries of the second row $b_k$.  Similarly, let $\binom{a_k''}{b_k''}$ be the tableau obtained from $\binom{a_k}{b_k}$ by removing the entries $a_k^1, ..., a_k^{h_1 + r - 1}$ from the top row $a_k$, removing the entries $b_k^1, ..., b_k^{h_1 - 1}$ from the bottom row $b_k$, and left-justifying the result.  Clearly, $\binom{a_k}{b_k}$ is obtained from $\binom{a_k'}{b_k'}$ and $\binom{a_k''}{b_k''}$ by concatenation of the rows.  The previous paragraph shows, however, that $\wc \binom{a_k}{b_k}$ is obtained similarly from $\wc \binom{a_k'}{b_k'}$ and $\wc \binom{a_k''}{b_k''}$ by concatenation of the rows.  As $C_k$ holds for $\binom{a_k}{b_k}$, the analogous statement holds for $\binom{a_k''}{b_k''}$, and as the row lengths of $\binom{a_k''}{b_k''}$ are strictly smaller than the row lengths of $\binom{a_k}{b_k},$ the statement $D_k$ follows by induction.\end{proof}

For convenience let us also extend Definition \ref{M} so that $M(a^i_k)=a^i_k=\alpha^i_k$ when $a^i_k$ is a hole, even though $a^i_k$ is not swapped with an entry in the row $b_k$.

When $A_k$ holds, and hence $B_k, C_k$, and $D_k$ as well, the the holes in $\binom{a_k}{b_k}$ remain invariant under $\wc$.  In particular, when $A_k$ holds, the $\mathfrak{wc}$ procedure, which includes both the swapping and the sorting, can be visualized as drawing $d$ non-intersecting two-way arrows between the $b^j_k$ and $a^{i_j}_k$ and swapping the $b^j_k$ according to the two-way arrows.

\begin{definition}
Define $E_k$ to be the following statement: the behavior of $\mathfrak{wc}$ acting on the two-row tableau $\binom{a_k}{b_k}$ is equivalent to drawing the $d$ non-intersecting two-way arrows between the $b^j_k$ and $a^{i_j}_k$ and swapping the $b^j_k$ according to the two-way arrows. 
\end{definition}

We have seen that $A_k$ implies $E_k$ for all $k \geq 0$.  Clearly, $E_k$ also implies $A_k$ for all $k$.

\begin{definition}
For integers $k \geq 0$ such that $E_{k'}$ holds for all nonnegative integers $k' \leq k$, we make the following definitions:
\begin{enumerate}
\item For integers $k \geq 0$ such that $E_{k'}$ holds for all nonnegative integers $k' \leq k$, let the $k'$-\emph{displacement} $D(j,k')$ be the horizontal length of the two-way arrow representing the swap $M(b^j_{k'}) = a_{k'}^{j + D(j, k')}.$ 

\item For all integers $k \geq 0$, let $F_k$ be the statement that $E_{k'}$ holds for all nonnegative integers $k' \leq k$ and that $b^i_{k'+1}\ge a^{i+D(i,k')}_{k'+1}$ for all $i.$

\item For all integers $k \geq 1$, let $G_k$ be the statement that $E_{k'}$ holds for all nonnegative integers $k' \leq k$ and that $D(j,k-1)\le D(j,k)$ for all $j$.

\end{enumerate}\end{definition}

It is clear that $D(j,k)\ge 0$ and that $D(1,k)\le \cdots \le D(d,k)$ because the two-way arrows are non-intersecting.

\begin{remark}
$E_0$ holds by inspection of the two-row tableau associated to $|\lambda, (0, \frac{e}{2})\rangle$\end{remark}

\begin{lemma}\label{f-implies-g-lemma}
$F_k$ implies $G_{k+1}.$
\end{lemma}
\begin{proof}  Suppose $F_k$ holds.  By definition of $F_k$, we have that $b^i_{k+1}\ge a^{i+D(i,k)}_{k+1}$ is true for all $i$ for which the expressions are defined.  We need to show that $D(i,k+1)\ge D(i,k)$ for all $i$.  Assume for contradiction that for some $i,$ we have $D(i,k+1)< D(i,k).$ Since $b^i_{k+1}\ge a^{i+D(i,k)}_{k+1},$ then $a^{i+D(i,k)}_{k+1}$ cannot be a hole and must be swapped with some element during the procedure $\wc$, i.e. we have $M(b_{k+1}^{i'})=a^{i+D(i,k)}_{k+1}$ for some index $i'.$ By Remark \ref{meta-increasing} we must have $i'>i.$ But now, since $a^{i'+D(i',k)}_{k+1}>a_{k+1}^{i+D(i,k)}$ and $a^{i'+D(i',k)}_{k+1}\le b_{k+1}^{i'}$ by hypothesis, $a^{i'+D(i',k)}_{k+1}$ must be swapped.  By induction, this procedure produces an infinite increasing list of entries from the second row $b_k$, which is a contradiction as the tableau $\binom{a_k}{b_k}$ is finite.  It follows that $D(i, k + 1) \leq D(i, k)$ for all $i$, as needed.\end{proof}

The following definition will be useful terminology in the final step of the proof of Theorem \ref{combinatorial-theorem}:

\begin{definition}In general, let us refer to the leftmost $ke+\frac{e}{2}$ indices of the first row of $\binom{a_k}{b_k}$ as the \emph{filler} of $\binom{a_k}{b_k}.$\end{definition}

\begin{lemma}\label{g-implies-f-lemma} $G_k$ implies $F_k$ for all $k\ge 1.$ Also, $E_0$ implies $F_0.$
\end{lemma}

\begin{proof}
We want to prove that $b^i_{k+1}\ge a^{i+D(i,k)}_{k+1}.$ By Remark \ref{delta-e}, we know that $b^i_{k+1}=\beta^i_k.$ We first consider a special case when $a^{i+D(i,k)}_{k+1}$ lies in the filler of $\binom{a_{k+1}}{b_{k+1}},$ in which case, by Remark \ref{delta-e}, $a^{i+D(i,k)}_{k+1}=a^{1}_{k+1}+i+D(i,k)-1.$ If $a^{i+D(i,k)}_{k+1}$ is not in the filler, we will have $a^{i+D(i,k)}_{k+1}=\alpha^{i+D(i,k)-e}_k+e.$ 
\hfill \break
\begin{description}
\item [Case 0]
$a^{i+D(i,k)}_{k+1}$ is in the filler of $\binom{a_{k+1}}{b_{k+1}}.$ If $i+D(i,k)=1,$ then clearly $b^i_{k+1}\ge a^{1}_{k+1}=a^{i+D(i,k)}_{k+1}.$ Otherwise $$a^{i+D(i,k)}_{k+1}=a^{1}_{k+1}+i+D(i,k)-1=a^{1}_k+i+D(i,k)-1$$ since $a^1=b^1$ is invariant under $\mathfrak{wc}$ and $\Delta_e.$ By definition of $D(i,k),$ we also know $b^i_{k+1}=\beta^i_{k}=a^{i+D(i,k)}_k.$  It then follows that $$b^i_{k+1}=a^{i+D(i,k)}_k\ge a^{1}_{k}+i+D(i,k)-1=a^{i+D(i,k)}_{k+1}.$$
\end{description}
\hfill \break
Outside of the filler we will have $a^{i+D(i,k)}_{k+1}=\alpha^{i+D(i,k)-e}_k+e.$ We consider two cases.
\hfill \break
\begin{description}
\item[Case 1] $\alpha^{i+D(i,k)-e}_k=a^{i+D(i,k)-e}_k$ is hole. Then $$b^i_{k+1}=\beta^i_k=a_k^{i+D(i,k)}\ge a_k^{i+D(i,k)-e}+e=\alpha_k^{i+D(i,k)-e}+e=a^{i+D(i,k)}_{k+1},$$ so we are done.
\hfill \break
\item[Case 2] Otherwise $\alpha^{i+D(i,k)-e}_k=b^{i+D(i,k)-e-\delta}_k$ for some $k$-displacement $\delta.$ Note that $\delta\le D(i,k)$ because the displacement lengths associated to all of the non-holes in the top row are decreasing from right to left. Note also that $\beta^i_k=a^{i+D(i,k)}_k.$ We therefore need to show $a^{i+D(i,k)}_k\ge b^{i+D(i,k)-e-\delta}_k+e.$ We consider two cases.
\hfill \break
\item[Case 2.1] $k=0.$ Recall that $\binom{a_0}{b_0}$ is the two-row tableau associated to a charged symmetric bipartition. Because $b^{i+D(i,0)-e-\delta}_0$ is well-defined, it follows that $i+D(i,0) \geq e > \frac{e}{2}.$ Hence $a^{i+D(i,0)}_0$ is not in the filler of $\binom{a_0}{b_0}$. Thus $$a^{i+D(i,0)}_0=b^{i+D(i,0)-\frac{e}{2}}_0+\frac{e}{2}\ge b^{i+D(i,0)-e}_0+e\ge b^{i+D(i,0)-e-\delta}_0+e,$$ and we are done.
\hfill \break
\item[Case 2.2] Otherwise $k\ge 1.$

Assume for contradiction that $a^{i+D(i,k)}_k$ is in the filler of $\binom{a_k}{b_k},$ which has length $ke+\frac{e}{2}.$ Then $i+D(i,k)\le ke+\frac{e}{2}.$ But since $a^{i+D(i,k)}_{k+1}$ is not in the filler of $\binom{a_{k+1}}{b_{k+1}},$ then $i+D(i,k)> (k+1)e+\frac{e}{2},$ a contradiction.

We conclude that $a^{i+D(i,k)}_k$ is not in the filler of $\binom{a_k}{b_k}.$ Then the desired inequality $a^{i+D(i,k)}_k\ge b^{i+D(i,k)-e-\delta}_k+e$ is equivalent to the inequality $\alpha_{k-1}^{i+D(i,k)-e}+e\ge \beta^{i+D(i,k)-e-\delta}_{k-1}+e.$
\hfill \break
\hfill \break
We consider two cases.
\hfill \break
\item[Case 2.2.1] $\alpha_{k-1}^{i+D(i,k)-e}$ is a hole. Then $\alpha_{k-1}^{i+D(i,k)-e}=a_{k-1}^{i+D(i,k)-e}.$  By $G_k,$ $\beta^{i+D(i,k)-e-\delta}_{k-1}=a^{i+D(i,k)-e-\delta+\delta'}_{k-1}$ for some $\delta'\le \delta.$  As $a_{k-1}^{i+D(i,k)-e}\ge a^{i+D(i,k)-e-\delta+\delta'}_{k-1},$ the desired inequality $\alpha_{k-1}^{i+D(i,k)-e}+e\ge \beta^{i+D(i,k)-e-\delta}_{k-1}+e$ follows, as needed.
\hfill \break
\item[Case 2.2.2] Otherwise $\alpha_{k-1}^{i+D(i,k)-e}$ is not a hole.  By $G_k$, we have $\alpha_{k-1}^{i+D(i,k)-e}=b_{k-1}^{i+D(i,k)-e-\delta'_1}$ for some $\delta'_1\le \delta$ and $\beta^{i+D(i,k)-e-\delta}_{k-1}=a_{k-1}^{i+D(i,k)-e-\delta+\delta'_2}$ for some $\delta'_2\le \delta$ and $\delta'_2\le \delta'_1.$ We need to show $b_{k-1}^{i+D(i,k)-e-\delta'_1}\ge a_{k-1}^{i+D(i,k)-e-\delta+\delta'_2}.$ As $$M\left(b_{k-1}^{i+D(i,k)-e-\delta}\right)=a_{k-1}^{i+D(i,k)-e-\delta+\delta'_2},$$ it follows that $b_{k-1}^{i+D(i,k)-e-\delta} \geq a_{k-1}^{i+D(i,k)-e-\delta+\delta'_2}$ by $C_{k - 1}$ (note that $G_k$ implies $E_{k - 1}$ which in turn implies $A_{k - 1}$ and hence $C_{k - 1}$).  In particular, we have $$b_{k-1}^{i+D(i,k)-e-\delta'_1}\ge b_{k-1}^{i+D(i,k)-e-\delta}\ge a_{k-1}^{i+D(i,k)-e-\delta+\delta'_2},$$ completing this case and the proof.
\end{description}
\end{proof}

\begin{proof}[Proof of Theorem \ref{combinatorial-theorem}] Statement $E_0$ is true, so by Lemma \ref{g-implies-f-lemma} $F_0$ is true as well.  By Lemma \ref{f-implies-g-lemma} and Lemma \ref{g-implies-f-lemma} it follows that $F_k$ holds for all integers $k \geq 0$ and that $G_k$ holds for all integers $k \geq 1$.  In particular $B_k$ holds for all $k \geq 0$, and by Lemma \ref{reduction-lemma} this completes the proof of Theorem \ref{combinatorial-theorem}.\end{proof}

\begin{remark}
The inductive argument used to prove Theorem \ref{combinatorial-theorem} depended on only the following base conditions of the two-row tableau $\binom{a_0}{b_0}$ representing the symmetric charged bipartition $|(\lambda^1, \lambda^1),(0,\frac{e}{2})\rangle$:

\begin{enumerate}
\item $A_0$ (hence $F_0$).
\item For all $i,$ $a^{i+D(i,0)}_0\ge b^{i+D(i,0)-e}_0+e.$
\end{enumerate}
 In terms of general charged bipartitions (where $\lambda^1$ is not necessarily equal to $\lambda^2$ and $s=(s_1,s_2)$ is not necessarily equal to $(0,\frac{e}{2})$), condition (1) means that the parts of the charged bipartition $|\lambda,s\rangle$, where $s_1\le s_2$, satisfy $\lambda_i^1\ge \lambda^2_{i+s_2-s_1}$ for all $i.$
\end{remark}

\end{document}